\documentclass[a4paper,11pt]{amsart}

\usepackage{amssymb}

\usepackage{youngtab}

\usepackage[hmargin=3.5cm,foot=0.7cm]{geometry}
\usepackage[utf8x]{inputenc}

\usepackage[english]{babel}

\usepackage{tikz}
\usetikzlibrary{matrix,arrows}

\newtheorem{theorem}{Theorem}[section]
\newtheorem*{theorem*}{Theorem}

\theoremstyle{plain}
\newtheorem{corollary}[theorem]{Corollary}
\newtheorem{lemma}[theorem]{Lemma}
\newtheorem{proposition}[theorem]{Proposition}

\theoremstyle{definition}
\newtheorem{definition}[theorem]{Definition}

\newtheorem*{remark}{Remark}

\newtheorem{problem}[theorem]{Problem}

\newcommand{\CC}{\mathbb{C}}

\newcommand{\RR}{\mathbb{R}}

\newcommand{\calH}{\mathcal{H}}
\newcommand{\calK}{\mathcal{K}}

\DeclareMathOperator{\Val}{Val}

\DeclareMathOperator{\vol}{vol}

\DeclareMathOperator{\Grass}{Gr}

\DeclareMathOperator{\Klain}{Kl}

\newcommand{\largewedge}{\mbox{\Large $\wedge$}}

\DeclareMathOperator{\conv}{conv}
\DeclareMathOperator{\pos}{pos}

\makeatletter
\DeclareRobustCommand{\pder}[1]{%
  \@ifnextchar\bgroup{\@pder{#1}}{\@pder{}{#1}}}
\newcommand{\@pder}[2]{\frac{\partial#1}{\partial#2}}
\makeatother

\newcommand{\der}[2][]{\frac{d#1}{d#2}}

\title[]{On the  extendability by continuity of angular valuations on  polytopes}

\author{Thomas Wannerer}
\email{thomas.wannerer@uni-jena.de}
\address{Fakult\"at f\"ur Mathematik und Informatik, Friedrich-Schiller-Universit\"at Jena, 07743 Jena, Germany}
\date{\today}
\thanks{Supported by DFG grant WA 3510/1-1.}
\begin{document}

\begin{abstract}
A classical theorem of P.~McMullen describes all valuations on polytopes that are 
invariant under translations and  weakly continuous, i.e., continuous with respect to parallel displacements of the facets of a polytope. While it is typically  not difficult to check that a valuation is weakly continuous, it is not  clear how to decide whether it admits a continuous extension to convex bodies.
In a special case of McMullen's construction a simple necessary and sufficient condition on the initial data of such an extension is obtained. 
\end{abstract}

\maketitle

\section{Introduction}

A  function on the space of  convex bodies in  $\RR^n$  
is called a valuation if it satisfies 
\begin{equation}\label{eq:add}\mu(K\cup L) = \mu(K)+\mu(L)-\mu(K\cap L)\end{equation}
 whenever $K\cup L$ is convex. The prime examples of valuations on convex bodies are the intrinsic volumes, but also many other of the most central constructions in Convex Geometry are valuations. Remarkably, in many  cases \eqref{eq:add}  together with a short list of inevitable properties determines these constructions uniquely. The seminal result in this direction is Hadwiger's characterization of the intrinsic volumes \cite{Hadwiger:Vorlesungen}; see, e.g., \cite{HaberlParapatits:Surface,Ludwig:Ellipsoids,Ludwig:Minkowski,Ludwig:Intersection,LudwigReitzner:Classification,SchusterW:Generalized,Schneider:AreaMeasures,Schneider:CurvatureMeasures} for results in similar spirit. The recent work of Alesker on valuations \cite{Alesker:Irreducibility,Alesker:Lefschetz,Alesker:Product,Alesker:Survey,Alesker:VMfdsIG,Alesker:Fourier} completely transformed this classical line of research by introducing numerous new tools and  opening up  new directions that extend well  beyond Convexity.

The most important examples of valuations on convex bodies are usually invariant under translations and continuous with respect to the Hausdorff metric. The restriction of  such a valuation to polytopes   satisfies \eqref{eq:add} for polytopes, is invariant under translations, and is  weakly continuous, i.e., continuous with respect to parallel displacements of the facets of a polytope. While it is typically  not difficult to  check that a valuation is weakly continuous, it is not clear how to decide whether it admits a continuous extension to convex bodies  (see, e.g., \cite{alesker:extendability,hhw}, \cite[Problem 2.3.13]{fu:barcelona}, and also \cite{HugSchneider:LocalTensor,HugSchneider:RotationCovariant} where this is a key issue):

\begin{problem}
Specify (simple necessary and sufficient) conditions under which a translation-invariant, weakly continuous valuation on polytopes admits a continuous extensions to  convex bodies.
\end{problem}

\begin{remark}
Since polytopes are a dense subset, an extension, if it exists, is unique. Moreover, by a theorem of Groemer \cite{Groemer:Extension}, a continuous extension is automatically  a valuation. 
\end{remark}

\subsection{Statement of the theorem}
According to McMullen \cite{mcmullen:weakly}, every translation-invariant, weakly continuous valuation on polytopes can be obtained by a fairly explicit geometric construction. A special case of this construction is the following. 
Let $F_{k,k+1}$ be the manifold of pairs $(E,\ell)$ where $E\subset \RR^n$ is a linear subspace of dimension $k$ and $\ell \subset \RR^n$ is a $1$-dimensional linear subspace orthogonal to $E$. For every continuous function $h\colon F_{k,k+1}\to \CC$ define
\begin{equation}\label{eq:muf} \mu_h(P)= \sum \int_{\ell \subset N_FP } h(\overline F, \ell ) \; d\ell, \end{equation} 
where  $P\subset \RR^n$ is a polytope, the sum extends over all $k$-dimensional faces $F$ of $P$, $\overline F$ is the translate of the affine hull  of $F$ containing the origin,  $N_FP$ is the normal cone of $P$ at the face $F$, $d\ell$ is the rotation-invariant probability measure on the space of rays in $\overline{F}^\perp$ emanating from the origin, and 
$h(\overline F, \cdot )$ is considered as an even function on rays in $\overline{F}^\perp$. By the easy part of McMullen's theorem \cite{mcmullen:weakly}, $\mu_h$ is a translation-invariant, weakly continuous valuation on polytopes.

It may sound surprising, but, as was shown by Alesker \cite{alesker:extendability},  there exist smooth functions $h$ such that $\mu_h$ does not admit  a continuous extension to convex bodies. Later, several  authors arrived at the same conclusion in very different ways  \cite{af:lorentz,hhw,pw:klain}. In the special case $n=3$ and $k=1$ Alesker~\cite{alesker:extendability} described a subspace of smooth sections  of a certain line bundle over $F_{1,2}$ that can be identified with those functions $h$ for which a continuous extensions is possible.  A sufficient condition for all $n$ and $k$ was found by Hinderer, Hug, and Weil \cite{hhw}: if $h$ lies in the image of a certain integral transform, then $\mu_h$ extends by continuity. 

The simplest non-trivial case of the construction \eqref{eq:muf} is when the function $h$ does not depend on $\ell$. In this situation, the integrals are proportional to the external angle $\gamma(F,P)$ of the  polytope $P$ at the face $F$,  
\begin{equation}\label{eq:angular}\mu_{\pi^* f}(P) = \sum f(\overline{F}) \gamma(F,P) \vol_k(F),\end{equation}
where  $\pi \colon F_{k,k+1}\to \Grass_k(\RR^n)$, $(E,\ell)\mapsto E$ is the projection to the Grassmannian and 
$f\colon \Grass_k(\RR^n)\to \CC$.
If $f\equiv 1$ one obtains the intrinsic volumes. Other important examples of such valuations that admit a continuous extension are  the Hermitian intrinsic volumes arising in Hermitian Integral Geometry \cite{hig} or more generally the constant coefficient valuations (see, e.g., \cite{hig,fu:barcelona}). Moreover,  valuations of type \eqref{eq:angular} are of interest in connection with the Angularity Conjecture \cite{bfs,w:angularity}.

The problem of characterizing those functions $f$ for which \eqref{eq:angular} admits a continuous extension  was explicitly stated by Fu   \cite[Problem 2.3.13]{fu:barcelona}.  Although \eqref{eq:angular} is arguably the simplest possible instance of \eqref{eq:muf}, the results of \cite{alesker:extendability,hhw} provide almost no information on this problem.   The purpose of this note is to resolve this situation by specifying a simple necessary and sufficient condition under which a continuous extension exists.

Throughout this article we assume that $n\geq 3$. 
\begin{theorem}\label{thm:main}
 Let $1\leq k <n-1$ and  $f\colon \Grass_k(\RR^n)\to \CC$. Then $\mu_{\pi^* f}$ admits a continuous extension to convex bodies if and only if $f$ is the restriction of a $2$-homogeneous polynomial to the image of the Pl\"ucker embedding. 
 Consequently, the space of such valuations has dimension
$$\frac{1}{n-k+1}\binom{n}{k}\binom{n+1}{k+1}$$
and coincides with the space of constant coefficient valuations.
\end{theorem}

\begin{remark} As was already observed in \cite{alesker:extendability} for $k=n-1$ there is no condition on $f$ except continuity.
Theorem~\ref{thm:main} directly implies the classification of angular curvature measures that has been recently obtain by the author \cite{w:angularity}.
\end{remark}

Let us clarify what we mean by restriction to the image of the Pl\"ucker embedding in the statement of theorem. The Pl\"ucker embedding  is a map \begin{equation}\label{eq:pluecker}\psi\colon \widetilde \Grass_k(\RR^n) \to \largewedge^k \RR^n\end{equation}
from the Grassmannian of oriented $k$-planes in $\RR^n$ into  the $k$th exterior power of $\RR^n$. If $u_1,\ldots, u_k$ is a positively oriented orthonormal basis of $E$, then 
$$\psi(E)= u_1\wedge \cdots \wedge u_k.$$
We call a function on the oriented Grassmannian even if it is invariant under change of orientation; even  functions on the oriented Grassmannian correspond bijectively to functions on the Grassmannian $\Grass_k( \RR^n )$ of (unoriented) linear subspaces of $\RR^n$. In particular one may restrict even functions on  $\largewedge^k \RR^n$
to the image of the Pl\"ucker embedding to obtain functions on the Grassmannian $\Grass_k(\RR^n)$.

\subsection{Outline of the proof}

 We conclude this introduction with a plan of the proof of Theorem~\ref{thm:main}. We denote by $\Val_k^+=\Val^+_k(\RR^n)$ the space of $k$-homogeneous, continuous valuations on convex bodies in $\RR^n$ that are invariant under translations and even (i.e., invariant under reflections in the origin).

\begin{definition}\label{def:angular}
 A valuation $\mu\in \Val_k^+$ is a called \emph{angular} if 
 $$\mu(P)= \sum \Klain_\mu(\overline F) \gamma (F,P) \vol_k(F)$$
 for every polytope $P\subset \RR^n$. Here the sum is over the $k$-dimensional faces $F$ of $P$,
 $\Klain_\mu\colon \Grass_k(\RR^n)\to \CC$ is the Klain function of $\mu$  (see Section~\ref{sec:prelim} for its definition),
 $\overline F$ is the translate of the affine hull of $F$ that contains the origin, and $\gamma(F,P)$ is the external angle of $P$ at $F$.  The subspace of angular valuations is denoted by $A_k\subset \Val_k^+$. 
\end{definition}

Clearly, if the valuation $\mu_{\pi^*f}$ extends by continuity to convex bodies, then  $\mu_{\pi^*f} \in A_k$
 and $\Klain_{\mu_{\pi^*f}} = f$. The strategy of the proof of Theorem~\ref{thm:main} is to  characterize the Klain functions of the elements of $A_k$. 
 To this end we pass to the more manageable subspace of smooth valuations introduced by Alesker. This will result in no loss of generality, since, as we will prove  in Section~\ref{sec:smooth}, smooth valuations are dense in $A_k$. By a deep theorem of Alesker, the Irreducibility Theorem \cite{Alesker:Irreducibility}, every smooth valuation can be represented by integration of a differential form over the normal cycle of a convex body (see Section~\ref{sec:prelim} below for precise statements). This provides us with the computational framework to deduce by an variational argument the  following necessary condition on the Klain function of $\mu\in A_{n-2}$. 
 Here and in the following we will not notationally  distinguish between functions on $\Grass_k(\RR^n)$ and even functions on unit length  simple $k$-vectors of $\RR^n$.

\begin{proposition} \label{prop:relation} If $\mu\in A_{n-2}$, then the function $f\colon \Grass_2(\RR^n) \to 
\RR$, 
$$f(E)=\Klain_\mu(E^\perp),$$ satisfies 
 \begin{equation} \label{eq:relation} (n-1)\int_{E_2^{n-1}} f( 
  \frac{1}{\sqrt{n-1}}\sum_{i=1}^{n-1}  \epsilon_i u_i \wedge u_n) \; d\epsilon   = 
\sum_{i=1}^{n-1} f( u_i \wedge u_n) \end{equation}
for every orthonormal basis $u_1,\ldots, u_{n}$ of $\RR^n$, where $d\epsilon$ denotes the uniform probability measure on $E_2^{n-1}=\{-1,1\}^{n-1}$. 
\end{proposition}

One may deduce a similar condition for $\mu\in A_{k}$. We will however proceed differently and first show that \eqref{eq:relation} holds if and only if $f$ is the  restriction of a $2$-homogeneous polynomial to the image of the Pl\"ucker  embedding and then finish the general case by an inductive argument.

\section{Preliminaries}
\label{sec:prelim}

\subsection{Convex geometry}

Let us denote by $\calK(\RR^n)$ the space of convex bodies in $\RR^n$. Continuity of functions on this space will always be understood with respect to the Hausdorff metric. The convex hull of a subset $A\subset \RR^n$ will be denoted by $\conv A$. For other basic notions from Convex Geometry not defined here we refer the reader to our standard reference \cite{Schneider:BM}.

\subsubsection{Continuous valuations}
Alesker's book  \cite{Alesker:Kent} and Chapter~6 of \cite{Schneider:BM} are our references for this material.
Let $\Val= \Val(\RR^n)$ denote the space of translation-invariant continuous valuations on $\RR^n$. By McMullen's decomposition theorem, 
$$ \Val = \bigoplus_{k=0}^n\Val_k,$$ 
where $\Val_k$ denotes the subspace of $k$-homogeneous valuations, i.e., those valuations $\mu$ satisfying $\mu(tK) = t^k\mu(K)$ for all $t>0$ and $K\in \calK(\RR^n)$. Moreover, there is the splitting $\Val_k= \Val_k^+\oplus \Val_k^-$ into even and odd valuations with respect to reflections in the origin. The space $\Val$ is equipped with the topology of uniform convergence on compact subsets. In fact, $\Val$ becomes a Banach space with respect to the norm 
$$ \| \mu \| = \sup_{K \subset B} |\mu(K)|,$$
where  $B$ is a convex body with nonempty interior.

The affine surface area is a notable exception of a geometrically relevant valuation that does not fit into this framework \cite{BesauLudwigWerner:Weighted,LudwigReitzner:Affine}; of a different and more recent vintage are valuations invariant under the indefinite orthogonal groups \cite{BernigFaifman:Indefinite,Faifman:Crofton}.

Let $\mu\in \Val_k^+$ and let $E\subset \RR^n$ be a $k$-dimensional linear subspace. By Hadwiger's characterization of volume (see, e.g., \cite[Theorem 9.1.1]{KlainRota:GP}), the restriction of $\mu$ to $E$ is proportional to the $k$-dimensional volume on $E$,
$$  \mu|_E = \Klain_\mu(E) \vol_k.$$
The map $\Klain\colon \Val_k^+ \to C(\Grass_k(\RR^n))$ is called the  Klain embedding and $\Klain_\mu$ is called the Klain function of $\mu$.  By a theorem of Klain \cite{Klain:Even}, $\Klain_\mu=0$ implies $\mu=0$.

\subsubsection{Convex cones}
If  $C\subset \RR^n$ is a closed convex cone, then the largest linear subspace contained in  $C$ is called the lineality space of $C$ and denoted by $L(C)$. The polar cone to $C$ is 
$$C^\circ =\{ x\in \RR^n\colon \langle x,y\rangle \leq 0 \text{ for } x\in C\}.$$
The external angle $\gamma(C)$ of $C$ is defined to be the fraction of $L(C)^\perp$ taken up by $C^\circ$. More precisely, 
$$ \gamma(C) = \frac{\vol_{n-k-1}(S^{n-1} \cap C^\circ)}{\vol_{n-k-1}(S^{n-k-1})}, \qquad \text{where } k=\dim L(C).$$
It will be important for us that the external angle of $C$ does not depend on the ambient space: If $C\subset \RR^m\subset \RR^n$, then  $\gamma(C)$ is the same whether computed in $\RR^m$ or in $\RR^n$.  

The tangent cone of a polytope $P$ at $x\in P$ is the closed convex cone
$$T_x P = \mathrm{cl}\bigcup_{t>0} t(P-x)$$
The lineality space of   $T_x P$ is the linear subspace parallel to the affine hull the unique face $F$ of $P$ 
containing $x$ in its relative interior. Note also that $T_xP = T_y P $ if $x,y$ belong to the relative interior of the same face $F$;
we denote this common cone by $T_FP$. The normal cone of $P$ at the face $F$ is 
$$ N_F P = T_F P ^ \circ.$$ 
The external angle of a polytope $P$ at a face $F$ is denoted $\gamma(F,P)=\gamma(T_FP)$.

If $\mu$ is a valuation on convex bodies in $\RR^n$ and $E\subset \RR^n$ is a linear subspace,
we denote by $\mu|_E$ the restriction of $\mu$ to $E$. 
\begin{lemma}\label{lemma:res} Let $E\subset \RR^n$ be a linear subspace. If $\mu\in A_k(\RR^n)$, then $\mu|_E\in A_k(E)$. 
\end{lemma}
\begin{proof}
 This follows immediately from $\Klain_\mu|_{\Grass_k(E)}
 =\Klain_{\mu|_E}$ and the fact that  the external angle $\gamma(F,P)=\gamma(T_F P)$ of $P\subset E \subset \RR^n$ is the same whether computed in $E$ or $\RR^n$.  
\end{proof}

\subsection{Smooth valuations and the normal cycle}

The general linear group $GL(n)$ acts on $\Val$ by $(g\cdot \mu)(K)= \mu(g^{-1} K)$. A valuation 
$\mu\in\Val$ is called smooth, if the map $GL(n)\to \Val$, $g\mapsto g\cdot \mu$ is smooth.

Let us write $S\RR^n=\RR^n\times S^{n-1}$ for the sphere bundle of $\RR^n$. 
As a set, the normal cycle of a convex body $K\subset \RR^n$ is  the set of outward unit normals to $K$,
$$N(K)=\{ (x,v)\in S\RR^n\colon v \text{ is an outward unit normal at } x\in K\}.$$
 Using the metric projection $p_K\colon \RR^n \to K$ (see \cite{Schneider:BM}) it is not difficult to see that $(x,v)\to x+v$ is a bi-Lipschitz equivalence between  $N(K)$ and the boundary of the  convex body $K+B^n$, where $B^n$ is the Euclidean unit ball. We orient normal cycle such that this is map preserves orientations; our convention for the orientation of a domain $D\subset \RR^n$ is that $v_1,\ldots,v_{n-1}\in T_p \partial D$ is positively oriented if $\nu_p,v_1,\ldots, v_{n-1}$ is, where $\nu_p$ is the outward normal at the point $p\in \partial D$.

 In the language of Geometric Measure Theory the preceding paragraph  can be summarized by saying that the normal cycle of a convex body is an integral current. It was proved in \cite{Alesker:VMfdsIII} 
that for a smooth $(n-1)$-form $\omega$ on $S\RR^{n-1}$ which is invariant under translations by elements of $\RR^n$ the function
$$K\mapsto \int_{N(K)} \omega$$ 
defines a smooth valuation. The converse is also true and follows from Alesker's Irreducibility theorem:

\begin{theorem}[{\cite[Theorem 5.2.1]{Alesker:VMfdsI}}]
 Let $0\leq k\leq n-1$. If $\mu\in \Val_{k}^\infty$ then there exists a translation-invariant smooth differential form $\omega\in \Omega^{n-1}(S\RR^n)$ of bi-degree $(k,n-k-1)$ such that 
 $$\mu(K)= \int_{N(K)} \omega$$
 for every convex body $K\subset \RR^n$. 
\end{theorem}

\subsection{Grassmannians and exterior powers}

Here we  recall the correspondence between simple $k$-vectors and $k$-planes in $\RR^n$. For more information and proofs see, e.g., \cite{Federer:GMT}. 
If $w_1,\ldots, w_k$ is a basis of a linear subspace  $E\subset \RR^n$, then 
$$w_1\wedge \cdots \wedge w_k$$
is a simple $k$-vector. A different choice of basis yields a proportional $k$-vector and 
$$ \{ v\in \RR^n\colon v\wedge w_1\wedge \cdots \wedge w_k=0\}= E.$$

A choice of orientation on $\RR^n$ is equivalent to a choice unit $n$-vector $\omega$. With respect to this choice the Hodge star operator $* \colon\largewedge^k\RR^n\to \largewedge^{n-k}\RR^n$ is defined by 
$$  \xi\wedge * \eta= \langle \xi, \eta\rangle \omega, \qquad \xi,\eta\in \largewedge^k\RR^n,$$
where the inner product $\langle \xi,\eta\rangle$ is defined by declaring $e_{i_1}\wedge \cdots \wedge e_{i_k}$, $1\leq i_1<\cdots < i_k\leq n$, for the the standard orthonormal basis $e_1,\ldots,e_n$ of $\RR^n$,  to be an orthonormal basis of $\largewedge^k\RR^n$. 
Hence, if the orientation on $E^\perp$ is chosen such that $E\oplus E^\perp =\RR^n$   has orientation $\omega$, then 
\begin{equation} \label{eq:hodge} \psi(E^\perp) = * \psi(E), \qquad E\in \widetilde \Grass_k(\RR^n), \end{equation}
where $\psi$ denotes the Pl\"ucker embedding \eqref{eq:pluecker}.

Finally, let us remark that if $v_1\wedge \cdots \wedge v_k$ and $w_1\wedge \cdots \wedge w_k $ are simple $k$-vectors then 
$$\langle v_1\wedge \cdots \wedge v_k, w_1\wedge \cdots \wedge w_k \rangle = \det (\langle v_i, w_j\rangle ).$$
In particular, 
\begin{equation}\label{eq:norm_volume} |v_1\wedge \cdots \wedge v_k| = k! \vol_k(\conv\{0,v_1,\ldots, v_k\}).\end{equation}

\subsection{Representation theory of the orthogonal group}

We recall from \cite{FultonHarris} the parametrizations of the  irreducible complex representations of the special  orthogonal groups $SO(n)$.
The description of the representations of these groups is sensitive to the parity of the dimension $n$.
The irreducible representations of $SO(2m+1)$ are parametrized by integer sequences $(\lambda_1, \ldots, \lambda_m)$ with \begin{equation}\label{eq:parametr_so_odd}\lambda_1\geq\ldots\geq \lambda_m\geq 0;\end{equation}
the irreducible representations of $SO(2m)$ are parametrized by integer sequences $(\lambda_1, \ldots, \lambda_m)$ with \begin{equation}\label{eq:parametr_so_even}\lambda_1\geq\ldots\geq \lambda_m\quad \text{and} \quad\lambda_{m-1}\geq |\lambda_m|.\end{equation}

For example,  the representation $\largewedge^i V$, where $V=\CC^n$ is the standard representation of $SO(n)$, is irreducible for $1\leq i < m$  and  corresponds to  
$$\lambda_1=\cdots=\lambda_i=1 \qquad \text{and}\qquad \lambda_{i+1}=\cdots =\lambda_{m}=0.$$
When  $n$ is odd, then also $\largewedge^m V$ is irreducible and has highest weight
$$\lambda_1=\cdots=\lambda_m=1;$$ 
but when $n$ is even  $\largewedge^m V$ is not irreducible: the eigenspaces of the Hodge star operator are the irreducible subrepresentations with highest weights 
$$\lambda_1=\cdots=\lambda_{m-1}=1 \quad \text{and} \qquad \lambda_{m}=\pm 1.$$ 

\subsubsection{Spherical harmonics} Let  $\calH_p^n$  denote the subspace of $p$-homogeneous complex-valued polynomials on $\RR^n$  that are harmonic, i.e., satisfy $\Delta f= 0$. The restriction of $\calH_p^n$ to $S^{n-1}$ yields precisely the irreducible subrepresentations of $C(S^{n-1})$ under the action of the orthogonal group.  Since 
$ \frac{\partial^2}{\partial x^2} + \frac{\partial^2}{\partial y^2}= 4 \frac{\partial^2}{\partial z \partial \bar z}$, the polynomial  
\begin{equation}\label{eq:sphericalhw}\bar z_1^p = (x_1 -ix_2)^p\end{equation}
is in $\calH_p^n$ and a short computation shows that it is in fact a highest weight vector of  weight 
$(p,0,\ldots, 0)$, see, e.g., \cite[Exercise 7.2]{Sepanski:LieGroups}.

\subsubsection{Restriction to the image of the Pl\"ucker embedding}

\begin{lemma}\label{lemma:highest_weights_restrictions}
 The subspace of $C(\widetilde\Grass_k(\RR^n))$ of restrictions of $2$-homogeneous polynomials on $\largewedge^k \RR^n$  to $\widetilde \Grass_k(\RR^n)$ decomposes under the natural action of 
 $SO(n)$  precisely into the  irreducible representations with highest weights $(2 m_1, 2 m_2,\ldots, 2 m_{\lfloor n/2\rfloor})$ satisfying 
 $|m_{1}|\leq 1$ and $m_i=0$ for $i>\min(k,n-k)$. Its dimension is 
 \begin{equation}\label{eq:plucker dim} \frac{1}{n-k+1}\binom{n}{k}\binom{n+1}{k+1}.\end{equation}
 \end{lemma}
\begin{proof}
 This is a well-known fact that is proved for example in \cite[Lemma 4.3]{w:angularity} and \cite[Lemma 4.2]{w:angularity}
\end{proof}

\section{Reduction to smooth valuations}
\label{sec:smooth}

\begin{lemma}
 $A_k\subset \Val_k^+$ is closed and invariant under the action of $O(n)$. 
\end{lemma}
\begin{proof}
 If $(\mu_i)_{i=1}^\infty$ is a sequence of angular valuations converging uniformly on compact subsets to some valuation $\mu\in \Val_k^+$, then also $\Klain_{\mu_i}\to \Klain_\mu$ uniformly. If $P$ is a polytope, then
 $$  \mu_i(P)= \sum \Klain_{\mu_i}(\overline F) \gamma (F,P) \vol_k(F)$$
 for every $i$. Passing to the limit, we see that $\mu$ is angular as well.   
It is clear form the definition that for every $g\in O(n)$ the valuation $g\cdot \mu$ is angular if $\mu$ is.
\end{proof}

Let us recall some terminology from representation theory (see, e.g., \cite{Knapp:Examples,Sepanski:LieGroups}).
Let $K\times V\to V$ be a representation of a group $K$ on an arbitrary vector space $V$. A vector $v\in V$ is a called $K$-finite the span of $\{ k\cdot v\colon k\in K\}$ is finite-dimensional. For a compact Lie group $K$, let $\widehat K$ be the set of equivalence classes of irreducible finite-dimensional representations of $K$.  A continuous representation of $GL(n)$ on a Frech\'et space $V$ is called admissible if for each $\gamma\in \widehat{O(n)}$ the sum of all irreducible $O(n)$-invariant subspaces that are in the class $\gamma$ has finite dimension. The existence of the  Klain embedding $\Klain\colon \Val_k^+\to C(\Grass_k(\RR^n))$ implies that the representation of $GL(n)$ on $\Val_i^+$ is admissible \cite[Proposition 4.1]{Alesker:McMullen}.

\begin{lemma}\label{lemma:smooth}
 $A_k\cap \Val_k^\infty$ is a dense subspace of $A_k$. 
\end{lemma}
\begin{proof}
The action of $O(n)$ on $A_k$ is a continuous representation of a compact Lie group on a Banach space. It is well known that in this situation the subspace of $O(n)$-finite vectors is dense (see, e.g., \cite[Theorem 3.46]{Sepanski:LieGroups}). It is another well-known fact that for every admissible representation of $GL(n)$ on a Frech\'et space, every $O(n)$-finite vector is smooth (see, e.g., \cite[Proposition 8.5]{Knapp:Examples}). We conclude that smooth valuations in $A_k$ form a dense subspace. 
\end{proof}

\section{Proof of Proposition~\ref{prop:relation}}

Throughout  this section let  $ \mu\in \Val_{n-2}^+(\RR^n)$ be fixed. Let us also fix an orientation $or$ of $\RR^n$ and let us  suppose that  $\mu$ is given by a translation-invariant smooth differential form $\omega\in \Omega^{n-1}(S\RR^n)$ of bi-degree $(n-2,1)$, 
 $$\mu(K)=\int_{N(K)} \omega.$$

 Consider the simplex $S\subset \RR^{n}$ spanned by the points $p_0=0$, $p_i=v_i$, $i=1,\ldots, n-1$, $p_n=t v_n$, where $t$ is a positive number, and $v_1,\ldots,v_n$ is some orthonormal basis of $\RR^n$. 
 Let us  make the dependence of the simplex $S$ on the choice of orthonormal basis more explicit by writing 
 $ S=S(v_1,\ldots, v_n)$. The first goal of this section is to prove the following lemma.
 Here and in the following for every function $f\colon \Grass_k(\RR^n)\to \CC$ and every simple nonzero $k$-vector $w_1\wedge \cdots \wedge w_k$ we will often write 
 $ f(w_1\wedge \cdots \wedge  w_k)$ instead of $f(E)$, $E=\{v\in \RR^n\colon v\wedge w_1\wedge \cdots \wedge  w_k=0\}$. 
 \begin{lemma}\label{lemma:comp1}
    $$\lim_{t\to 0^+}  \der{t}  \int_{E_2^{n-1}}  \mu(S (\epsilon_1v_1,\ldots, \epsilon_{n-1}v_{n-1}, v_n)) \; d\epsilon  = \frac{1}{4(n-2)!} \sum_{1\leq i< j<n}  \Klain_\mu( \bigwedge_{\substack{1\leq k\leq n\\k\neq i,j}} v_k ).$$
 \end{lemma}

 We will introduce now some notation that will be useful not only in the proof of Lemma~\ref{lemma:comp1}, but also in the proof of Lemma~\ref{lemma:comp2} below.
 Let  $F_0,F_1,\ldots, F_n$ be the facets of $S$, where the notation is chosen such that $p_i\notin F_i$. Note that every $(n-2)$-face of $S$ is of the form 
 $$ F_i\cap F_j = \conv\{ p_k\colon k\neq i,j\}.$$
 The dihedral angle $\theta_{ij}\in (0,\pi)$ of this face is given by 
 $$ \cos\theta_{ij} = \langle u_i,u_j\rangle,
 $$
 where $u_i$ is the facet normal of $F_i$. Note that 
 $$u_i=-v_i, \qquad i=1,\ldots, n,$$
 and 
 $$u_0=  \frac{1}{\sqrt{1+(n-1)t^2}} (\sum_{i=1}^{n-1} tv_i + v_n).$$

 \begin{lemma}\label{lemma:volume}
  Let $v_1,\ldots, v_n$ be an orthonormal basis of $\RR^n$. Then 
  $$ \vol_{n-1}( \conv\{ v_1,\ldots, v_n\})=  \frac{\sqrt n}{(n-1)!}.$$
 \end{lemma}
\begin{proof}
 Note that 
 $$ \vol_{n-1}( \conv\{ v_1,\ldots, v_n\}) = \frac{h}{n-1} \vol_{n-2}(\conv\{ v_1,\ldots, v_{n-1}\}),$$
 where $h$ is the distance from $v_n$ to the affine subspace spanned by $ v_1,\ldots, v_{n-1}$. 
 Since 
 \begin{equation}\label{eq:ortho_decomp}v_n=  \frac{1}{n-1}(v_1 + \cdots +v_{n-1}) +  (  v_n - \frac{1}{n-1}(v_1 + \cdots +v_{n-1}))\end{equation}
 where the first vector is in the affine hull of  $v_1,\ldots, v_{n-1}$ and the second vector is orthogonal to that affine subspace, we get 
 $$ h^2  =  \frac{n}{n-1}.$$
 Induction on $n$ finishes the proof.
\end{proof}

 We divide the set of all $(n-2)$-faces of $S$ into four disjoint subsets. 
 
 \begin{enumerate}
  \item $ F_i\cap F_n$ for  $i=1,\ldots,n-1$. 
 These faces  do not depend on the parameter $t$ and have volume 
 $$ \vol_{n-2}(F_i\cap F_n) = \frac{ 1}{(n-2)!}.$$
 The normal cones   
 $$ N_{F_i\cap F_n}S = \pos \{ -v_i,-v_n\}$$
 are constant and $\theta_{in}=\pi/2$.
 
 \item  $ F_i\cap F_j$ for distinct $i,j=1,\ldots n-1$. 
 These faces  depend on $t$, but their affine hull does not.
Their  volume  is 
 $$ \vol_{n-2}(F_i\cap F_j) = \frac{ t}{(n-2)!}.$$
The normal cones  
 $$ N_{F_i\cap F_j}S  = \pos \{ -v_i,-v_j\}$$
are constant and $\theta_{ij}=\pi/2$
 
 
 \item $ F_0\cap F_n$.
 This face does not depend on the parameter $t$, but its normal cone does. By Lemma~\ref{lemma:volume} we have 
 $$   \vol_{n-2}(F_0\cap F_n) = \frac{ \sqrt{n-1}}{(n-2)!}. $$
 The normal cone is 
  $$ N_{F_0\cap F_n}S = \pos \{ u_0,-v_n\}$$ and the dihedral angle is
 $$\theta_{0n}(t) = \arccos (- \frac{1}{\sqrt{1+(n-1)t^2}}).$$
 Moreover
 $$S^{n-1} \cap N_{F_0\cap F_n} S  = \{ \cos\alpha u_n + \sin\alpha w_n \colon 0\leq \alpha\leq \theta_{0n}\},$$
 where 
 $$ w_n=\frac{1}{\sqrt{n-1}} \sum_{i=1}^{n-1} v_i. $$
 
 
 \item $F_0\cap F_i$ for $i=1,\ldots, n-1$. These faces 
 have nonconstant volume, affine hull, normal cone, and dihedral angle.
 It follows from a decomposition similar to \eqref{eq:ortho_decomp} that 
 $$ \vol_{n-2}(F_0\cap F_i) = \frac{ \sqrt{1+(n-2)t^2}}{(n-2)!}.$$
 The normal cone is 
  $$ N_{F_0\cap F_i}S = \pos \{ u_0,-v_i\}$$ and the dihedral angle is
 $$\theta_{0i}(t) = \arccos (- \frac{t}{\sqrt{1+(n-1)t^2}}).$$
 Moreover
 $$S^{n-1} \cap N_{F_0\cap F_i} S  = \{ \cos\alpha u_i + \sin\alpha w_i \colon 0\leq \alpha\leq \theta_{0i}\},$$
 where 
 $$ w_i=\frac{1}{\sqrt{1 +(n-2)t^2 }} (\sum_{\substack{1\leq j\leq n-1\\ j\neq i}} tv_j + v_n) . $$

 \end{enumerate}
 
\begin{proof}[Proof of Lemma~\ref{lemma:comp1}]
Observe that 
$$\mu(S)= \sum_F \int_{F\times (N_F S \cap S^{n-1})} \omega,$$
where the sum is over all $(n-2)$-faces of $S$. Using the preceding division of the $(n-2)$-faces 
of $S$, up to an additive constant that does not depend on $t$ we have that    $\mu(S(v_1,\ldots, v_n))$  equals $S_1+S_2+S_3$, where  
 \begin{align*}
  & S_1:=\sum_{1\leq i< j<n} (-1)^{i+j+n+1}\frac{ t}{(n-2)!} or(v_1 \wedge \cdots  \wedge v_n)  
  \\ & \hspace{4cm} \int_{0}^{\pi/2} \langle \omega(-\cos\alpha v_i -\sin \alpha v_j ), \bigwedge_{
   \substack{1\leq k\leq n\\k\neq i,j}} v_k \wedge 
  (\sin\alpha \bar v_i -\cos \alpha \bar v_j )\rangle \, d \alpha,   \\
  & S_2:=  \frac{1}{(n-2)!}  or(v_1\wedge \cdots \wedge v_n)
  \\ &  \hspace{2.5cm} \int_0^{\theta_{0n}(t)}  \langle \omega (\cos\alpha u_n + \sin\alpha w_n), \bigwedge_{1<k<n}(v_k-v_1) \wedge (-\sin\alpha \bar u_n + \cos\alpha \bar w_n)  \rangle  \, d\alpha, \\
  &S_3:= \sum_{i=1}^{n-1} \frac{(-1)^i}{(n-2)!} or(v_1\wedge \cdots \wedge v_n) \\
  &  \hspace{2.5cm} \int_0^{\theta_{0i}(t)}  \langle \omega (\cos\alpha u_i + \sin\alpha w_i(t)  ) ,\bigwedge_{\substack{1\leq k<n\\k\neq i}}(v_k-tv_n)\wedge  (-\sin\alpha \bar u_i + \cos\alpha \bar w_i ) \rangle \,  d\alpha.
 \end{align*}
Here we always  write  $\bar x =(0,x)\in \RR^n\oplus \RR^n$ and $x= (x,0)\in  \RR^n\oplus \RR^n$ if we form wedge products.  Moreover, we have used \eqref{eq:norm_volume} to obtain
$$ | \bigwedge_{1<k<n}(v_k-v_1) | = (n-2)! \vol_{n-2}(\conv \{ v_1,\ldots, v_{n-1}\}) = (n-2)! \vol_{n-2}(F_0\cap F_n)$$
and 
$$|\bigwedge_{\substack{1\leq k<n\\k\neq i}}(v_k-tv_n)|=(n-2)! \vol_{n-2}(F_0\cap F_i)$$
for $i=1,\ldots, n-1$.
 
Now  we compute the first derivative with respect to $t$ and let $t\to 0^+$.
Let us make the dependence of $S_1(t)$ on the choice of orthonormal basis explicit by writing  $S_1(t)(v_1,\ldots, v_n)$. We have 
\begin{align*}  & \int_{E_2^{n-1}} S'_1(t)( \epsilon_1 v_1,\ldots, \epsilon_{n-1} v_{n-1}, v_n)\, d\epsilon \to \sum_{1\leq i<j<n}  (-1)^{i+j+n+1}  or( v_1\wedge \cdots \wedge v_n ) \\
 &  \hspace{2cm}  \frac{1}{4(n-2)!}  \int_0^{2\pi}  \langle  \omega(\cos\alpha v_i + \sin\alpha v_j),\bigwedge_{
   \substack{1\leq k\leq n\\k\neq i,j}} v_k\wedge (-\sin\alpha \bar v_i + \cos\alpha \bar v_j) \rangle \\
 &  = \frac{1}{4(n-2)!} \sum_{1\leq i< j<n}  \Klain_\mu( \bigwedge_{\substack{1\leq k\leq n\\k\neq i,j}} v_k ).
\end{align*}

Next we use 
\begin{equation}\label{eq:theta0n}\theta_{0n}(t)\to \pi, \qquad   \theta_{0n}'(t)\to -\sqrt{n-1}\end{equation}
as $t\to 0^+$. Since the integrand does not depend on $t$, taking derivatives yields 
$$ S_2'(t)\to  \frac{\sqrt{n-1}}{(n-2)!}  or(v_1\wedge \cdots \wedge v_n) \langle \omega(v_n),\bigwedge_{1<k<n}(v_k-v_1) \wedge \bar w_n\rangle.  $$
Now 
\begin{align*}  \sqrt{n-1} \bigwedge_{1<k<n}(v_k-v_1) \wedge \bar w_n  & = (-1)^n \sum_{k=1}^{n-1} v_1 \wedge \cdots \wedge \bar v_k\wedge \cdots \wedge v_{n-1}+ R, 
\end{align*}
where $R$ contains all terms that are multiples of  $v_i\wedge \bar v_i$ for some $1\leq i \leq n-1$. After integration over $E_2^{n-1}$ the terms coming from $R$ cancel and we get
\begin{align*}  \int_{E_2^{n-1}}  &  S'_2(t)( \epsilon_1 v_1,\ldots, \epsilon_{n-1} v_{n-1}, v_n) \, d\epsilon \to\\
& \frac{ (-1)^n}{(n-2)!}  or(v_1\wedge \cdots \wedge v_n)   \langle \omega(v_n), 
\sum_{k=1}^{n-1} v_1 \wedge \cdots \wedge \bar v_k\wedge \cdots \wedge v_{n-1}\rangle.
\end{align*}

Finally, since 
\begin{equation}
 \label{eq:theta0i} \theta_{0i}(t)\to \pi/2, \qquad \theta_{0i}'(t)\to 1, \qquad i=1, \ldots, n-1,
\end{equation}
we get
\begin{align*} & S'_3(t)\to \sum_{i=1}^{n-1} \bigg(  \frac{(-1)^{n+1}}{(n-2)!}or(v_1\wedge \cdots \wedge v_n)   \langle \omega( v_n ), v_1 \wedge \cdots \wedge \bar v_i\wedge \cdots \wedge v_{n-1}  \rangle \\
&+ \frac{(-1)^i}{(n-2)!}   or(v_1\wedge \cdots \wedge v_n)  \int_0^{\pi/2} \langle \langle \operatorname{grad} \omega(-\cos\alpha v_i + \sin \alpha v_n ), \sin\alpha \sum_{\substack{1\leq j\leq n-1\\ j\neq i}} v_j\rangle ,\\
 & \hspace{3cm}\bigwedge_{\substack{1\leq k<n\\k\neq i}}v_k\wedge  (\sin\alpha \bar v_i + \cos\alpha \bar v_n  \rangle  d\alpha \\
& +  \frac{ (-1)^i}{(n-2)!} or(v_1\wedge \cdots \wedge v_n)   \int_0^{\pi/2} \langle \omega(-\cos\alpha v_i + \sin \alpha v_n ), \xi \rangle    d\alpha \bigg),
\end{align*}
where 
\begin{align*}\xi:=&\left. \der{t}\right|_{t=0} \bigwedge_{\substack{1\leq k<n\\k\neq i}}(v_k-tv_n)\wedge  (-\sin\alpha \bar u_i + \cos\alpha \bar w_i(t))  \\
 =& \sum_{\substack{1\leq k<n\\k\neq i}}(-1)^{k+i+1}    v_1\wedge \cdots  \wedge  \underbrace{(\sin \alpha \bar v_i+ \cos\alpha \bar v_n)}_{i\text{th position}} \wedge \cdots  \wedge  
\widehat{v_k} \wedge \cdots \wedge v_n \\
&  \hspace{1cm}+ \cos\alpha   v_1\wedge \cdots  \wedge \widehat{v_i} \wedge \cdots  \wedge  
v_{n-1} \wedge \bar v_k,
\end{align*}
where $\widehat{v_k}$ means that this factor is omitted. 
If we integrate over $E_2^{n-1}$, then the first summand will cancel the contributions from $S_2(t)$ and the second summand will vanish. Indeed,  the term
$$   \int_0^{\pi/2} \langle \langle \operatorname{grad} \omega(-\cos\alpha v_i + \sin \alpha v_n ),  v_j\rangle ,\\
 \bigwedge_{\substack{1\leq k<n\\k\neq i}}v_k\wedge  (\sin\alpha \bar v_i + \cos\alpha \bar v_n  \rangle  d\alpha,$$
where $1\leq j\leq n-1$, $j\neq i$,  will not  change sign if we replace $v_j$ by $-v_j$. 
Similarly after integration the third summand will vanish since 
\begin{align*}(-1)^{k+i+1}    v_1\wedge \cdots  \wedge  \underbrace{(\sin \alpha \bar v_i+ \cos\alpha \bar v_n)}_{i\text{th position}} & \wedge \cdots  \wedge  
\widehat{v_k} \wedge \cdots \wedge v_n\\
& + \cos\alpha   v_1\wedge \cdots  \wedge \widehat{v_i} \wedge \cdots  \wedge  
v_{n-1} \wedge \bar v_k, \end{align*}
where $1\leq k<n$, $k\neq i$,
will not  change sign if we replace $v_k$ by $-v_k$. 
\end{proof}

\begin{lemma} \label{lemma:comp2} If $\mu$ is angular, then 
  \begin{align*}
    & \lim_{t\to 0^+}  \der{t}  \int_{E_2^{n-1}} \mu(S (\epsilon_1v_1,\ldots, \epsilon_{n-1}v_{n-1}, v_n)) \,d\epsilon = \frac{1}{4(n-2)!} \sum_{1\leq i< j<n-1}  \Klain_\mu( \bigwedge_{\substack{1\leq k\leq n\\k\neq i,j}} v_k)\\
    &-\frac{1}{2\pi} \frac{n-1}{(n-2)!} \int_{E_2^{n-1}}  \Klain_\mu ( 
  \sum_{i=1}^{n-1} (-1)^i  \epsilon_i \bigwedge_{\substack{1\leq k<n\\k\neq i}}v_k)\,d\epsilon +
  \frac{1}{2\pi} \frac{1}{(n-2)!} \sum_{i=1}^{n-1} \Klain_\mu(\bigwedge_{\substack{1\leq k<n\\k\neq i}}v_k).  
  \end{align*}

\end{lemma}

\begin{proof}
 If $\mu$ is angular then, up to an additive  constant which does not depend on $t$, we have that  $\mu(S (v_1,\ldots, v_{n-1}, v_n))$ equals  $A_1+A_2+A_3$, where 
 \begin{align*}A_1(t)  :=& \frac{t}{4(n-2)!} \sum_{1\leq i< j<n-1}  \Klain_\mu( \bigwedge_{\substack{1\leq k\leq n\\k\neq i,j}} v_k ),\\
  A_2(t)  := &\frac{\theta_{0n}(t)}{2\pi} \frac{\sqrt{n-1}}{(n-2)!}  \Klain_\mu ( \bigwedge_{ 1<k<n} (v_k-v_1)),\\
  A_3(t)  :=&  \sum_{i=1}^{n-1}\frac{\theta_{0i}(t)}{2\pi} \vol_{n-2}(F_0\cap F_i(t))\Klain_\mu(F_0\cap F_i(t)) \\
    = &\sum_{i=1}^{n-1}\frac{\theta_{0i}(t)}{2\pi} 
   \frac{ (-1)^i}{(n-2)!}or(v_1\wedge \cdots \wedge v_n) \\ &  \hspace{2cm}  \int_0^{2\pi}   \langle \omega (\cos\alpha u_i + \sin\alpha w_i(t)  ) ,\bigwedge_{\substack{1\leq k<n\\k\neq i}}(v_k-tv_n)\wedge  (-\sin\alpha \bar u_i + \cos\alpha \bar w_i  \rangle  d\alpha.
 \end{align*}
 Taking the first derivative, letting $t\to 0^+$, and integrating over $E_2^{n-1}$,  the term
 $A_1(t)$ yields
 $$ \frac{1}{4(n-2)!} \sum_{1\leq i< j<n-1}  \Klain_\mu( \bigwedge_{\substack{1\leq k\leq n\\k\neq i,j}} v_k ).$$
 For $A_2(t)$ we use \eqref{eq:theta0n} and get 
 \begin{equation}\label{eq:A2}   -\frac{1}{2\pi} \frac{n-1}{(n-2)!}   \int_{E_2^{n-1}}  \Klain_\mu ( \bigwedge_{ 1<k<n} (\epsilon_k v_k- \epsilon_1 v_1))\,d\epsilon.\end{equation}
For $A_3(t)$ we see as in the proof of Lemma~\ref{lemma:comp1} that after integration over $E_2^{n-1}$ the first derivative of $$\vol_{n-2}(F_0\cap F_i(t))\Klain_\mu(F_0\cap F_i(t)) $$ vanishes. Using \eqref{eq:theta0i}  we are left with 
$$\frac{1}{2\pi} \frac{1}{(n-2)!} \sum_{i=1}^{n-1} \Klain_\mu(\bigwedge_{\substack{1\leq k<n\\k\neq i}}v_k).  $$
To simplify \eqref{eq:A2} note that 
$$ \bigwedge_{ 1<k<n} ( v_k-  v_1)=  \sum_{i=1}^{n-1} (-1)^{i+1} \bigwedge_{\substack{1\leq k<n\\k\neq i}}v_k$$ 
and hence \eqref{eq:A2} equals 
 $$-\frac{1}{2\pi} \frac{n-1}{(n-2)!} \int_{E_2^{n-1}}  \Klain_\mu ( 
  \sum_{i=1}^{n-1} (-1)^i  \epsilon_i \bigwedge_{\substack{1\leq k<n\\k\neq i}}v_k)\, d\epsilon.$$
\end{proof}

\begin{proof}[Proof of Proposition~\ref{prop:relation}]
 Let us first assume that $\mu\in A_{n-2}$ is smooth. From Lemma~\ref{lemma:comp1} and Lemma~\ref{lemma:comp2} we conclude that 
 $$ (n-1) \int_{E_2^{n-1}}  \Klain_\mu ( 
  \frac{1}{\sqrt{n-1}}\sum_{i=1}^{n-1} (-1)^i  \epsilon_i \bigwedge_{\substack{1\leq k<n\\k\neq i}}u_k) \, d\epsilon = 
  \sum_{i=1}^{n-1} \Klain_\mu(\bigwedge_{\substack{1\leq k<n\\k\neq i}}u_k) $$
  for every orthonormal basis $u_1,\ldots, u_{n}$ of $\RR^n$.
  Applying the Hodge star operator  yields
$$ * (-1)^i   \bigwedge_{\substack{1\leq k<n\\k\neq i}}u_k = (-1)^{n+1} u_i \wedge u_n.$$
Combining this with \eqref{eq:hodge}  we obtain \eqref{eq:relation}. Since by Lemma~\ref{lemma:smooth} any continuous angular valuation can be approximated by smooth angular valuations, the claim follows. 
\end{proof}

\section{Necessary condition for $k=n-2$.}

In the section we will deduce the necessary condition in Theorem~\ref{thm:main} in the special case $k=n-2$.

\begin{proposition} \label{prop:characterization}
 Let $f\colon \Grass_2 (\RR^n)\to \RR$ be a continuous function. Then $f$ satisfies the relation \eqref{eq:relation}
 for every orthonormal basis $u_1,\ldots,u_n$ of $\RR^n$ if and only if $f$ is the restriction of a $2$-homogeneous polynomial on $\largewedge^2 \RR^n$ to the image of the Pl\"ucker embedding. 
\end{proposition}

\begin{proof}
Let us first prove that the restriction of every to $2$-homogeneous polynomial satisfies \eqref{eq:relation}.
To see this recall that 
if $p\colon V \to \CC$ is a homogeneous polynomial of degree $k$ on  a finite-dimensional real vector space $V$, then there exists a 
symmetric multilinear function  $p\colon V^k\to \CC$, 
 called the polarization of $p$ and denoted by again by $p$,  which is uniquely determined by the equation
 $$ p(a,\ldots, a) = p(a),\qquad \text{for all}\ a\in V.$$
For a proof see, e.g., \cite[Appendix A]{Hormander:Convexity}. Now if $f$ is the restriction of a $2$-homogeneous polynomial $p$ on $\largewedge^2\RR^n$ to the image of the Pl\"ucker embedding, then expanding the left-hand side of \eqref{eq:relation} by the bilinearity of the polarization
of $p$ immediately yields the claim. 

To prove the converse, let $W\subset C(\Grass_2(\RR^n))$ be the subspace of all continuous functions $f$ on $\Grass_2(\RR^n)$ with the property that \eqref{eq:relation} holds for every orthonormal basis 
$u_1,\ldots, u_n$. 
 We have to show that $W$ is contained in  the space of restrictions of $2$-homogeneous polynomials on  $\largewedge^k \RR^n$ to  $\Grass_2(\RR^n)$.

Clearly,  the subspace $W$ is closed and  invariant under the action of $SO(n)$. Hence $W$ decomposes into certain irreducible subrepresentations.  The highest weights occurring in the representation $C(\Grass_2(\RR^n))$ are well
known. According to Strichartz \cite{strichartz}, if $n\geq 4$, they all have multiplicity $1$ and are of the form 
$(2m_1,\ldots, 2m_{2},0\ldots, 0)$ 
with integers satisfying 
$$m_1\geq   |m_{2}|\qquad \text{and} \qquad \text{if}\  n>4,\ m_2\geq 0.$$
By Lemma~\ref{lemma:highest_weights_restrictions}, the proof for $n\geq 4$ will be finished if we can show that only the highest weights with
 $$ m_1\leq 1$$
 can occur in $W$. For this it will suffice to show that for all other highest weights the corresponding highest weight vector $f_{m_1, m_{2}}$ is not an element of $W$. This step will be based on an 
 explicit description of the highest weight vectors given by Strichartz  \cite{strichartz} (see also \cite{Alesker:Lefschetz}).
 
For any subspace $E\in \Grass_2(\RR^n)$ choose an orthonormal basis $X^1,X^2$ of $E$ and consider the corresponding $n \times 2$ matrix 
 $$\begin{pmatrix}
    X^1_1  &  X^2_1\\
    \vdots & \vdots  \\
   X^1_{n}  &   X^2_{n}\\ 
   \end{pmatrix}
$$
of coordinates with respect to the standard basis $e_1,\ldots, e_{n}$ of $\RR^{n}$.
Let $X_j$ denote the $j$th row of this matrix. For $l\leq \lfloor n/2 \rfloor$ let $A(l)$ be the $l\times 2$ matrix whose $j$th row is $X_{2j-1} + 
\sqrt{-1} X_{2j}$, $j=1,\ldots, l$. Note that the $l\times l$ matrix  $A(l)A(l)^t$ is independent of the choice of the orthonormal basis of $E$.

A  highest weight vector of the irreducible subrepresentation of $C(\Grass_2(\RR^n))$ with highest weight $(2m_1, 2m_2,0,\ldots, 0)$  
is given by 
\begin{enumerate}
 \item if $m_2\geq 0$,  $$ f_{m_1, m_2}= \det\left( A(1)A(1)^t\right)^{m_1-m_{2}}\det\left( A(2)A(2)^t\right)^{m_2};$$ 
 \item if $m_2<0$,
 $$ f_{m_1, m_2}= \det\left( A(1)A(1)^t\right)^{m_1-|m_{2}|}\overline{\det\left( A(2)A(2)^t\right)}^{|m_2|}.$$
\end{enumerate}

Let $e_1,\ldots, e_n$ be the standard basis of  $\RR^n$ and put 
\begin{align*}
u_1& = se_1 -c e_3,\\
u_2& = e_2,\\
u_i& = e_{i+1}, \qquad 3\leq i\leq n-1,\\
u_n& = ce_1 +s e_3,
\end{align*}
where $c=\cos \phi$ and $s=\sin\phi$ for brevity.

For this choice of orthonormal basis let us   evaluate  \eqref{eq:relation}  for a highest weight vector $f=f_{m_1,m_2}$. Throughout we assume that $m_1\neq 0$. 
A straightforward computation shows that 
\begin{equation}\label{eq:rhs}\sum_{j=1}^{n-1} f_{m_1,m_2}( u_j \wedge u_n) =\begin{cases} 
                                                         1 + (-1)^{m_2} c^{2m_1} + (-1)^{m_1} s^{2m_1}, &  m_2\neq 0;\\
                                                         1 + (n-3) c^{2m_1} + (-1)^{m_1} s^{2m_1},&  m_2= 0,\\
                                                        \end{cases}
\end{equation}
and 
\begin{align}\notag (n-1)& \int_{E_2^{n-1}}    f( 
  \frac{1}{\sqrt{n-1}}\sum_{j=1}^{n-1}  \epsilon_j u_j \wedge u_n )  \,d\epsilon =  \\
 \label{eq:lhs} &   \frac{2^{m_2}}{(n-1)^{m_1-1}}  \int_{E_2^{n-1}} 
  ((n-2)c^2 + 2i \epsilon_1\epsilon_2 s  )^{m_1-m_2} (\epsilon_2\epsilon_3 cs + i \epsilon_1(\epsilon_2 s -\epsilon_3 c))^{m_2}\,d\epsilon , 
\end{align}
provided $m_2\geq0$; if $m_2<0$, then the second factor inside the integral has to be conjugated and $m_2$ replaced by $|m_2|$. 
We will distinguish the  cases $m_2=0$, $m_2$ even (and different from $0$), $m_2$ odd. 

Let us start with the case $m_2=0$. Note that by \eqref{eq:rhs} and \eqref{eq:lhs} both sides of \eqref{eq:relation}
are trigonometric polynomials of degree $2m_1$. Comparing the coefficients of the leading term $c^{2m_1}$, we get
$$  \frac{(n-2)^{m_1}}{(n-1)^{m_1-1}} =  (n-2).$$
This can only hold if $m_1=1$.

Next let us treat the case  $m_2$ even and different from $0$. Again by \eqref{eq:rhs} and \eqref{eq:lhs} the leading term on both sides is $c^{2m_1}$ and comparing coefficients we get 
$$ 2^{|m_2|} \frac{(n-2)^{m_1-|m_2|}}{(n-1)^{m_1-1}} =2.$$
This is equivalent to 
$$ \left( \frac{n-2}{n-1} \right) ^{m_1-1}  = \left( \frac{n-2}{2} \right) ^{|m_2|-1}.$$
Since the left-hand side is strictly less than $1$ and the right-hand side greater equal $ 1$, this is a contradiction. Hence no highest weight vector $f=f_{m_1,m_2}$ with 
$m_2\neq 0$ even can satisfy \eqref{eq:relation}.

Let us finally consider the case $m_2$ odd. Expanding  the two factors inside the  sum in \eqref{eq:lhs} we get 
\begin{align} \notag ((n-2)c^2)^{m_1-m_2} +\binom{m_1-m_2}{1} &  ((n-2)c^2)^{m_1-m_2-1} 2i \epsilon_1\epsilon_2 s  \\
  \label{eq:expansionI} &  + \binom{m_1-m_2}{2} ((n-2)c^2)^{m_1-m_2-2}  (2i  s)^2+ \cdots 
\end{align}
and 
\begin{align} \notag\epsilon_2\epsilon_3 (cs)^{m_2} +  & \binom{m_2}{1}  (cs)^{m_2-1}   i\epsilon_1(\epsilon_2 s- \epsilon_3 c)\\ 
 \label{eq:expansionII} & \hspace{2cm}+ 
\binom{m_2}{2} \epsilon_2\epsilon_3  (cs)^{m_2-2} (i\epsilon_1(\epsilon_2 s- \epsilon_3 c))^2 + \cdots
\end{align}
where $\cdots$ denotes lower order terms.  Observe that 
$$(i\epsilon_1(\epsilon_2 s- \epsilon_3 c))^2= 2\epsilon_2\epsilon_3 cs-1.$$
Integrating over $E_2^{n-1}$ we see that the leading term $c^{2m_1-2}$ on the left-side of \eqref{eq:relation} results from the product of the second terms in the above expansions and the product of the first in \eqref{eq:expansionI} with the third in \eqref{eq:expansionII}. Comparing coefficients we get 
$$ (-1)^{(m_2-1)/2} 
  \frac{2^{m_2+1}(n-2)^{m_1-m_2-1}}{(n-1)^{m_1-1}}  \big((n-2)\binom{m_2}{2}  + (m_1-m_2)m_2\big)
= - m_1.$$ 
Consequently, $(m_2-1)/2$ must be odd and in particular $m_2\geq 3$. 
If $n=2k+2$ is even, then
$$ 2^{m_1} k^{m_1-m_2-1} \big((n-2)\binom{m_2}{2}  + (m_1-m_2)m_2\big)= m_1 (2k+1)^{m_1-1}.$$
We conclude that $2^{m_1}$ divides $m_1$, a contradiction. If $n=2k+1$ is odd ($k\geq 2$), then 
\begin{equation}
 \label{eq:oddcase} (2k-1)^a \big((2k-1) \binom{m_2}{2}+  (a+1)m_2 \big)
  =  m_1 2^{a-1} k^{m_2+a},
\end{equation}
where we put $a= m_1-m_2-1$. 
Since $2l \leq k^l$ for integers $k,l\geq 2$, we have  
\begin{align*}  (2k-1)^a   \big( (2k-1)   \binom{m_2}{2} & +  (a+1)m_2 \big) \\
& \leq  m_2 (2k-1)^{a+1} \frac{k^{m_2-1}}{4} 
+ (2k-1)^a (a+1) m_2 \\
& < m_2 2^{a-1} k^{m_2+a} +  (2k-1)^a (a+1) m_2\\
& = m_1 2^{a-1} k^{m_2+a} + (a+1)( (2k-1)^a m_2 - 2^{a-1}  k^{m_2+a}) \\
& \leq m_1 2^{a-1} k^{m_2+a},
\end{align*}
which contradicts \eqref{eq:oddcase}. 
Therefore if $m_2$ is odd and \eqref{eq:relation} holds for $f_{m_1,m_2}$, then we must have $m_1-|m_2|=0$. 

Now if $m_1=|m_2|\geq 3$ is odd, then after integration over $E_2^{n-1}$ we see that  the leading terms in \eqref{eq:lhs} and \eqref{eq:rhs}  are $(cs)^{m_1-1}$ and $c^{2m_1-2}$. Comparing coefficients we get 
$$(-1)^{(m_1-1)/2}    \frac{2^{m_1}}{(n-1)^{m_1-1}} (m_1-1)   = -1.$$
For the signs to match we must have $m_1=4k+3$ with $k=0,1,2,\ldots$. Hence
\begin{equation}\label{eq:identityOdd}4 \cdot  2^{4k+2} (2k+1) = (n-1)^{4k+2}\end{equation}
which implies that $n-1$ must contain the prime factor $2$ more then once. But then 
\eqref{eq:identityOdd} can only hold if $k=0$ and $n-1=4$.

Now  for $n=5$ and $m_1=|m_2|=3$ one immediately checks that \eqref{eq:lhs} and \eqref{eq:rhs} do coincide. Hence we are going to choose a different orthonormal basis of $\RR^n$:
\begin{align*} u_1 & = c(ae_1+ be_3)+ s e_n,\\
  u_2 & = s(ae_1+be_3)-c e_n,\\
   u_3 & = e_2,\\
   u_i & = e_i, \qquad  4\leq i\leq n-1, \\
   u_n & = -be_1+ae_3,
\end{align*}
where $a=\cos\psi , b=\sin \psi $ for brevity.
A straightforward computation shows that 
\begin{equation}\label{eq:rhs5}\sum_{j=1}^{n-1} f_{m_1,\pm m_1}( u_j \wedge u_n) =c^{2m_1}+s^{2m_1} +(-1)^{m_1} (a^{2m_1}+b^{2m_1})                                                
\end{equation}
and 
\begin{align}\notag (n-1) \int_{E_2^{n-1}}  &  f_{m_1,\pm m_1}( 
  \frac{1}{\sqrt{n-1}}\sum_{j=1}^{n-1}  \epsilon_j u_j \wedge u_n )  \,d\epsilon = \\
 \label{eq:lhs5} & \hspace{1cm} \frac{2^{m_1}}{(n-1)^{m_1-1}}  \int_{E_2^{n-1}} 
( \epsilon_1\epsilon_2 cs - \epsilon_3\epsilon_4 ab + i (\epsilon_1 c +\epsilon_2 s) (\epsilon_3 a +\epsilon_4 b))^{m_1}\,d\epsilon.
\end{align}
Comparing the coefficients the leading term  we get a contradiction for $m_1>1$. We conclude that the validity of \eqref{eq:relation} for a highest weight vector $f_{m_1,m_2}$ implies $m_1\leq 1$.

It remains to treat the case $n=3$. Taking orthogonal complements, we have to show that if an even  continuous  function 
$f\colon S^2\to \CC$ satisfies 
$$  2\int_{E_2^2}  f(\frac{1}{\sqrt 2} \sum_{j=1}^2 \epsilon_j u_j)\,d\epsilon =\sum_{j=1}^2 f(u_j)$$
for every pair of orthonormal vectors $u_1,u_2$, then $f$ is the restriction of a $2$-homogeneous polynomial to $S^2$. According to \eqref{eq:sphericalhw} the function $f=f_{m_1}= (x_1-ix_2)^{2m_1}$, is contained  in the irreducible subrepresentation with highest weight $(2m_1)$. Now if $m_1>0$, then choosing as before $u_1=se_1-ce_3, u_2=e_2$, evaluating  both sides, and comparing the coefficients of the leading term  $s^{2m_1}$,  we get
$$ \frac{1}{2^{m_1-1}} =1. $$
Hence also in this case $m_1\leq 1$, as desired.
\end{proof}

\begin{corollary}\label{cor:codegree2} If  $\mu \in A_{n-2}(\RR^n)$, then  the Klain function of $\mu$ is the restriction of a $2$-homogeneous polynomial on $\largewedge^{n-2}\RR^n$ to the image of the Pl\"ucker embedding. 
\end{corollary}

\begin{proof} Proposition~\ref{prop:relation}, Proposition~\ref{prop:characterization}, and \eqref{eq:hodge} imply that if $\mu \in A_{n-2}$, then the Klain function of $\mu$ is the restriction of a $2$-homogeneous polynomial on $\largewedge^{n-2}\RR^n$ to the image of the Pl\"ucker embedding.

\end{proof}

\section{Proof of Theorem~\ref{thm:main}}

In this final section we provide an inductive argument that allows us to deduce Theorem~\ref{thm:main} from the special case $k=n-2$ already established in  Corollary~\ref{cor:codegree2}. We will need the following fact. 

\begin{proposition}[{\cite[Proposition 4.4]{w:angularity}}]\label{prop:polynomial_restriction}
 Let $1\leq k <n-1$ and $f$ be an even continuous function on $\widetilde{\Grass}_k(\RR^n)$. Suppose that for every nonzero $v\in \RR^n$ there exists a 
 $2$-homogeneous polynomial $q_v$ on $\largewedge^k v^\perp$ such that $f=q_v$ on $\widetilde{\Grass}_k(v^\perp)\subset \largewedge^k v^\perp$. Then there exists  
  a globally defined $2$-homogeneous polynomial $q$ on $\largewedge^k \RR^n$ such that 
 $f=q$ on $\widetilde{\Grass}_k(\RR^n)\subset \largewedge^k \RR^n$.
\end{proposition}

\begin{corollary}\label{corollary:induction}
  Let  $1\leq k <n-1$ and $f$ be an even continuous function on $\widetilde{\Grass}_k(\RR^n)$. Suppose that for every linear subspace  $E\subset  \RR^n$ of dimension $k+2$ there exists a 
 $2$-homogeneous polynomial $q_E$ on $\largewedge^k E$ such that $f=q_E$ on $\widetilde{\Grass}_k(E)\subset \largewedge^k E$. Then there exists  
  a globally defined $2$-homogeneous polynomial $q$ on $\largewedge^k \RR^n$ such that 
 $f=q$ on $\widetilde{\Grass}_k(\RR^n)\subset \largewedge^k \RR^n$.
\end{corollary}

\begin{proof} Put $m=n-k-2$. If $m=0$, then there is nothing to prove. Assume now that $m>0$ and that the theorem has been proved for $m-1$. Then for every nonzero vector $v\in \RR^n$ there exists a 
 $2$-homogeneous polynomial $q_v$ on $\largewedge^k v^\perp$ such that $f=q_v$ on $\widetilde{\Grass}_k(v^\perp)\subset \largewedge^k v^\perp$. By Proposition~\ref{prop:polynomial_restriction} this implies the existence of 
  a globally defined $2$-homogeneous polynomial $q$ on $\largewedge^k \RR^n$ such that 
 $f=q$ on $\widetilde{\Grass}_k(\RR^n)\subset \largewedge^k \RR^n$. 
\end{proof}

\begin{proof}[Proof of Theorem~\ref{thm:main}] 
If $\mu=\mu_{\pi^*f}$  admits a continuous extension to  convex bodies then $\mu\in A_k$.  
For every linear subspace $E\subset \RR^n$ the restriction $\mu|_E$ of $\mu$ to $E$ is according to Lemma~\ref{lemma:res}  angular. The Klain function of $\mu|_E$ is the restriction of the Klain function of $\mu$ to $\Grass_k(E)$. If the dimension of $E$ is $k+2$, then $f|_{\Grass_k(E)}=\Klain_{\mu|_E}$ is by Corollary~\ref{cor:codegree2} the restriction of a $2$-homogeneous polynomial on $\largewedge^k E$   to the image of the Pl\"ucker embedding. By  Corollary~\ref{corollary:induction}, $f$ is the restriction of $2$-homogeneous polynomial on $\largewedge^k \RR^n$ to the image of the Pl\"ucker embedding. 
 
The converse follows from the fact proved in \cite[Theorem 1.4]{w:angularity} that to each 
restriction of a $2$-homogeneous polynomial  to the image of the Pl\"ucker embedding corresponds a constant coefficient curvature measure. Globalizing this curvature measure we get an element of $A_{k}$. That   $A_k$ has the stated dimension is now immediate from  Lemma~\ref{lemma:highest_weights_restrictions}. 
\end{proof}

\bibliographystyle{abbrv}
\bibliography{ref_convex,ref_integral}
\end{document}